\documentclass{amsproc}

\usepackage{amssymb}

\usepackage{graphicx}

\usepackage[cmtip,all]{xy}

\usepackage[all]{xy}
\usepackage{amsthm}
\usepackage{amscd}
\usepackage{amsmath}
\usepackage{amsfonts}
\usepackage{amssymb}
\usepackage{stmaryrd}
\usepackage{graphicx}%
\usepackage{hyperref}
\hypersetup{
    bookmarks=true,         % show bookmarks bar?
    unicode=false,          % non-Latin characters in Acrobat¡¯s bookmarks
    pdftoolbar=true,        % show Acrobat¡¯s toolbar?
    pdfmenubar=true,        % show Acrobat¡¯s menu?
    pdffitwindow=false,     % window fit to page when opened
    pdfstartview={FitH},    % fits the width of the page to the window
    pdftitle={My title},    % title
    pdfauthor={Author},     % author
    pdfsubject={Subject},   % subject of the document
    pdfcreator={Creator},   % creator of the document
    pdfproducer={Producer}, % producer of the document
    pdfkeywords={keyword1, key2, key3}, % list of keywords
    pdfnewwindow=true,      % links in new PDF window
    colorlinks=false,       % false: boxed links; true: colored links
    linkcolor=red,          % color of internal links (change box color with linkbordercolor)
    citecolor=green,        % color of links to bibliography
    filecolor=magenta,      % color of file links
    urlcolor=cyan           % color of external links
}

\newtheorem{theorem}{Theorem}[section]
\newtheorem{lemma}[theorem]{Lemma}

\theoremstyle{definition}
\newtheorem{definition}[theorem]{Definition}
\newtheorem{example}[theorem]{Example}
\newtheorem{proposition}[theorem]{Proposition}

\theoremstyle{remark}
\newtheorem{remark}[theorem]{Remark}

\numberwithin{equation}{section}

\begin{document}

\title{Quasi-theories}


%    author one information

\author{Zhen Huan}

\address{Zhen Huan, Department of Mathematics,
Sun Yat-sen University, Guangzhou, 510275 China} \curraddr{}
\email{huanzhen84@yahoo.com}
%\thanks{The author was partially supported by NSF grant DMS-1406121.}

%    author two information
%\author{}
%\address{}
%\curraddr{} \email{}
%\thanks{}

%\subjclass[2000]{Primary }
%    The 2010 edition of the Mathematics Subject Classification is
%    now available.  If you are citing a classification from the
%    new scheme, use the following input coding instead.
\subjclass[2010]{Primary 55}

\date{}

\begin{abstract}
In this paper we define a family of theories, quasi-theories, motivated by quasi-elliptic cohomology.
They can be defined from constant loop spaces. With them, the constructions on certain theories can be made in a neat way, such as 
those on generalized Tate K-theories. We set up quasi-theories and discuss their properties.
\end{abstract}

\maketitle


\section{Introduction}
In \cite{Rez11} \cite{Huanthesis} \cite{Huansurvey} we set up quasi-elliptic cohomology $QEll^*_G(-)$. It is a variant of Tate K-theory, which is 
the generalized elliptic cohomology theory associated to the Tate curve $Tate(q)$ over Spec$\mathbb{Z}((q))$.  Quasi-elliptic cohomology is defined over Spec$\mathbb{Z}[q^{\pm}]$
and has a direct interpretation in terms of the Katz-Mazur group scheme $T$ [Section 8.7, \cite{KM85}]. Its relation with Tate K-theory is \begin{equation}QEll^*_G(X)\otimes_{\mathbb{Z}[q^{\pm}]}\mathbb{Z}((q))=(K^*_{Tate})_G(X).\end{equation}

We can go a step further and consider the $n-$th generalized Tate K-theory, for each positive integer $n$. It is defined over Spec$\mathbb{Z}((q))^{\otimes n}$ and its divisible group is $\mathbb{G}_m\oplus (\mathbb{Q}/\mathbb{Z})^n$. As shown in Example
\ref{generalizedquasi}, we can extend the idea of quasi-elliptic cohomology and construct a theory $QK_{n, G}^*(-)$ over Spec$\mathbb{Z}[q^{\pm}]^{\otimes n}$. It can be expressed 
in terms of equivariant K-theories. In addition, $$QK_{n, G}^*(X)\otimes_{\mathbb{Z}[q^{\pm}]^{\otimes n}}\mathbb{Z}((q))^{\otimes n}$$ is isomorphic to the $n-$th generalized Tate K-theory $(K^*_{n, Tate})_G(X)$.

As shown in Definition 2.14 and Definition 3.11 \cite{Huansurvey}, for each compact Lie group $G$ and $G-$space $X$, 
we can construct a groupoid $\Lambda(X/\!\!/G)$ consisting of constant loops in the free loop space $LX$. From it we define quasi-elliptic cohomology by 
\[QEll^*_G(X):=K^*_{orb}(\Lambda(X/\!\!/G)).\] In Definition \ref{orbifoldloopspace} we discuss the composition $\Lambda^n(-)$ of this constant loop functor. It defines
$$QK_{n, G}^*(X):=K^*_{orb}(\Lambda^n(X/\!\!/G)).$$ As shown in Section \ref{properties}, those constructions on quasi-elliptic cohomology can all be applied to $QK_{n, G}^*(-)$,
including restriction maps, K\"{u}nneth map, change-of-group isomorphism, etc.
In a further paper we will also construct the power operation of this theory and classify the level structure and finite subgroups of the corresponding divisible group. We did that work for quasi-elliptic cohomology theory in \cite{HuanPower}.

More generally, with the same loop space $\Lambda^n(X/\!\!/G)$, we  define a family of cohomology theories $QE_{n}^*(-)$ with the K-theory $K^*$ replaced by any other equivariant theory $E^*$. They are given the name quasi-theories.

We show in a coming paper \cite{Huanglobal} that if $E^*$ can be globalized, then $QE_n^*$ can be globalized. 

In Section \ref{loopmodel}, 
we study the composition of the free loop functor and that of the loop functors from which we construct quasi-elliptic cohomology in \cite{Huansurvey}. Especially, we give a loop space construction  of $G-$equivariant generalized Tate K-theory for compact Lie groups $G$.
In Section \ref{definequasi}, we define the quasi-theories and discuss some examples. In Section \ref{properties}, we present some properties of the quasi-theories.

\subsection{Acknowledgement}
First I would like to thank my PhD advisor Charles Rezk. He suggested this research project when I was a PhD student at UIUC. I would like to thank Marc Levine. His questions on quasi-elliptic cohomology during my visit at Universit$\ddot{a}$t Duisburg-Essen helped me understand the subject and its nature more deeply. I would like to thank Chenchang Zhu for showing me the relation between bibundles and stacks during her talk at Sun Yat-sen University and during my visit at Georg-August-Universit$\ddot{a}$t G$\ddot{o}$ttingen. And I would like to thank Nathaniel Stapleton for answering my questions nicely all the time.

\section{Loop space models}\label{loopmodel}

\subsection{The $n-$th loop space via bibundles}

In \cite{LerStack} Lerman discussed the construction of bibundles and showed the weak 2-category of Lie groupoids with 1-arrows bibundles can be embeded into 
the 2-category of stacks, which is a sub 2-category of the category of categories.
In \cite{Huansurvey} we construct a loop space via bibundles. In this paper we study the $n-$th power of it.
First we recall the definition of bibundles, which is Definition 3.25 \cite{LerStack}.
\begin{definition}[Bibundles]
Let $G$ and $H$ be two groupoids. A bibundle from $G$ to $H$ is a manifold $P$ together with two maps 
\[\xymatrix{G_0 &P\ar[l]_{a_L}\ar[r]^{a_R} &H_0}\] such that
\begin{itemize}
\item there is a left action of $G$ on $P$ with respect to $a_L$ and a right action of $H$ on $P$ with respect to $a_R$;

\item $a_L: P\longrightarrow G_0$ is a principal $H-$bundle;

\item $a_R$ is $G-$invariant: $a_R(g\cdot P)=a_R(p)$ for all $(g, p)\in G_1\times_{H_0}P$;

\item the actions of $G$ and $H$ commute.

\end{itemize}

\end{definition}
\begin{definition}[Bibundle maps]
A bibundle map is a map $P\longrightarrow P'$ over
$\mathbb{H}_0\times \mathbb{G}_0$ which commutes with the
$\mathbb{G}-$ and $\mathbb{H}-$actions. \end{definition}

For each pair of Lie groupoids $\mathbb{H}$ and $\mathbb{G}$,  we
have a category $Bibun(\mathbb{H}, \mathbb{G})$ with as objects
bibundles from $\mathbb{H}$ to $\mathbb{G}$ and as morphisms the
bundle maps. 
%The category of smooth functors from $\mathbb{H}$ to $\mathbb{G}$ is a subcategory of $Bibun(\mathbb{H}, \mathbb{G})$.

\begin{example}[$Bibun(S^1/\!\!/\ast, \ast/\!\!/G)$] A bibundle from $S^1/\!\!/\ast$
to $\ast/\!\!/G$ with $G$ a Lie group is a smooth manifold $P$
with $a_L: P\longrightarrow S^1$ a smooth
principal $G-$bundle and  $a_R: P\longrightarrow
\ast$ a constant map.  Thus,  a bibundle in this case is equivalent to a smooth
principal $G-$bundle over $S^1$. The morphisms in
$Bibun(S^1/\!\!/\ast, \ast/\!\!/G)$ are bundle isomorphisms.
\label{babyloop1}
\end{example}

In Definition 2.4 \cite{Huansurvey} we define a loop space via bibundles.
\begin{definition}[$Loop_1(X/\!\!/G)$] 
Let $G$ be a Lie group acting smoothly on a manifold $X$. The loop space $Loop_1(X/\!\!/G)$
is defined to be  the groupoid  $Bibun(S^1/\!\!/\ast,
X/\!\!/G)$.
\end{definition}

Only the $G-$action on $X$ is considered in $Loop_1(X/\!\!/G)$. We
add the rotations by adding more morphisms into $Loop_1(X/\!\!/G)$.

\begin{definition}[$Loop^{ext}_1(X/\!\!/G)$]\label{loopext3space} 
The loop space $Loop^{ext}_1(X/\!\!/G)$  is defined to be  the groupoid with the same
objects as $Loop_1(X/\!\!/G)$. Each morphism
consists of the pair $(t, \alpha)$ where $t\in\mathbb{T}$ is a
rotation and $\alpha$ is a morphism in $Loop_1(X/\!\!/G)$. They make the diagram
below commute.
$$\xymatrix{S^1\ar[d]_{t}
&P \ar[l]_{\pi}\ar[d]_{\alpha}\ar[r]^{f} & X \\S^1 &P'
\ar[l]^{\pi'}\ar[ru]_{f'} &}$$
\end{definition}

In this paper we study the $n-$th power of the functor $Loop_1(-)$ and $Loop^{ext}_1(-)$. In Definition 4.1 and 4.2 \cite{Huansurvey}
we show both of them can be defined for groupoids other than translation groupoids. So  $Loop^n_1(-)$ and $Loop^{ext, n}_1(-)$  are both well-defined.
\begin{definition}[$Loop^n_1(-)$ and $Loop^{ext, n}_1(-)$] 
The $n-$th power $Loop^n_1(-)$  is defined to be the composition of n $Bibun(S^1/\!\!/\ast, -)$ 
\[Bibun(S^1/\!\!/\ast, Bibun(S^1/\!\!/\ast, \cdots Bibun(S^1/\!\!/\ast, -))).\]

The $n-$th power $Loop^{ext, n}_1(-)$  is defined to be the composition of n $Loop^{ext}_1(-)$
\[Loop^{ext}_1(Loop^{ext}_1(\cdots Loop^{ext}_1(-))).\]

\end{definition}
\subsection{The $n-$th power of the free loop functor}

Before we introduce the orbifold loop space, we study first the $n-$th power $L^n(-)$ of the free loop functor $L(-)$.
For each $G-$space $X$, $L^nX$ is equipped with the action by $Aut(\mathbb{T}^n)$ and that by $L^nG$.

Let $G$ be a compact Lie group. Recall the free loop space of any space $X$
\begin{equation}LX:=\mathbb{C}^{\infty}(S^1, X). \end{equation} It comes with the action by the circle group $\mathbb{T}$. When $X$ is a $G-$space,
it has the action by the loop group $LG$.
The $n-$th power of the free loop functor
\begin{equation}L^nX=\mathbb{C}^{\infty}((S^1)^{\times n}, X)\end{equation} comes with an evident action by the torus group
$\mathbb{T}^n=(\mathbb{R}/\mathbb{Z})^{\times n}$ defined by %rotating the circle
\begin{equation}t\cdot \gamma:= (s\mapsto \gamma (s+t)), \mbox{
}t\in \mathbb{T}^n, \mbox{   } \gamma\in L^nX.\end{equation}
When $X$ is a $G-$space, $L^nX$  is  equipped with an action
by $L^nG$   \begin{equation}\delta\cdot
\gamma:=(s\mapsto \delta(s)\cdot \gamma(s)),\mbox{ for any } s\in
(S^1)^{\times n}, \mbox{   }\delta\in L^nX, \mbox{ }\mbox{
    }\gamma\in L^nG.\end{equation}

Combining the action by group of automorphisms $Aut((S^1)^{\times n})$ on the
torus and the action by $L^nG$, we get an action by the extended
$n-$th loop group $\Lambda^n G$ on $LX$. $\Lambda^n G:=L^nG\rtimes\mathbb{T}^n$
is a subgroup of
\begin{equation} L^nG\rtimes Aut((S^1)^{\times n}), \mbox{             }(\gamma, \phi)\cdot (\gamma', \phi'):= (s\mapsto \gamma(s) \gamma'(\phi^{-1}(s)), \phi\circ\phi')\end{equation}
with $\mathbb{T}^n$ identified with the group of rotations in $Aut(S^1)^{\times n}\leqslant Aut((S^1)^{\times n})$.
$\Lambda^n G$ acts on $L^nX$ by
\begin{equation}\delta \cdot(\gamma, \phi):= (t\mapsto
\delta(\phi(t))\cdot\gamma(\phi(t))), \mbox{  for any }(\gamma,
\phi)\in \Lambda^n G,\mbox { and    }\delta \in
L^nX.\label{loop2action}\end{equation} It's straightforward to check
(\ref{loop2action}) is a well-defined group action.

Then we define the corresponding twisted loop group. Let $$G^n_f$$ denote the set consisting of elements of the form $\sigma=(\sigma_1, \sigma_2, \cdots \sigma_n)$ where each $\sigma_i:\mathbb{R}^{n}\longrightarrow G$  is  a continuous map satisfying $\sigma_i(s_1, \cdots, s_{i-1}, -, s_{i+1},\cdots, s_n)$ is a constant function when $s_1, \cdots, s_{i-1},  s_{i+1},\cdots, s_n$ are fixed. There is a group structure on $G^n_f$ defined by \begin{equation}(\sigma'\cdot \sigma)(s):= (\sigma_1'(s)\sigma_1(s), \cdots \sigma_n'(s)\sigma_n(s)).\end{equation} The identity element is the $\sigma$ with each $\sigma_i$ a constant map to the identity element of $G$.

For each $\sigma\in G^n_f$, %and $l$ denote the order of $g$.
define the twisted loop group $L_{\sigma}G$ to be the group
\begin{equation}
\{\gamma: \mathbb{R}\longrightarrow
G|\gamma(s+e_i)=\sigma_i^{-1}(s)\gamma(s)\sigma_i(s)\mbox{,   for }i=1, \cdots n\}
\end{equation} where $e_i=(0,\cdots 0,1, 0,\cdots
0)\in \mathbb{R}^n$.
The multiplication of it is defined by
\begin{equation}(\delta\cdot\delta')(s)=\delta(s)\delta'(s)\mbox{,     for      any     }\delta, \delta'\in
L_{\sigma}G,  \mbox{    and     }s\in\mathbb{R}^{\times n}.\end{equation} The
identity element $e$ is the constant map sending $\mathbb{R}^{\times n}$ to the identity element of $G$. Similar to $\Lambda^n G$, we
can define $L_{\sigma}G\rtimes\mathbb{T}^n$ whose multiplication is
defined by
\begin{equation} (\gamma, t)\cdot (\gamma', t'):= (s\mapsto \gamma(s) \gamma'(s+t), t+t').\label{lkgtmulti}\end{equation}

In the case when each $\sigma_i: \mathbb{R}^{\times n}\longrightarrow G$ in $\sigma$ is a constant map, the set of constant maps $\mathbb{R}^{\times n}\longrightarrow G$ in
$L_{\sigma}G$ is a subgroup of it, i.e. the intersection of the centralizers $$\bigcap_{i=1}^n C_G(\sigma_i).$$

\subsection{Orbifold loop space}
In this section we discuss another model of loop space.

Let $G$ be a compact Lie group and $X$ be a $G-$space.

Recall in Definition 2.8 \cite{Huansurvey}, we defined the groupoid $Loop_2(X/\!\!/G)$.
\begin{definition}[$Loop_2(X/\!\!/G)$]\label{oldloopspace3}
Let
$Loop_2(X/\!\!/G)$ denote the groupoid whose objects are $(\sigma_1,
\gamma)$ with $\sigma_1\in G$ and $\gamma: \mathbb{R}\longrightarrow
X$
a continuous map such that $\gamma(s+1)= \gamma(s)\cdot\sigma_1$, for any $s\in\mathbb{R}$. %$\mbox{Map}([0, 1], X)$
A morphism $\alpha: (\sigma_1, \gamma)\longrightarrow (\sigma'_1,
\gamma')$ is a continuous map $\alpha: \mathbb{R}\longrightarrow
G$ satisfying $\gamma'(s)= \gamma(s)\alpha(s)$.
%Note that $\alpha(s)\sigma'=\sigma\alpha(s+1)$, for any $s\in\mathbb{R}$.
\end{definition}

The $n-$th power of the loop functor $Loop_2^n(-)$ is well-defined because the definition of $Loop_2(-)$ can be extended to any  groupoid.
We describe it explicitly in the example below.
\begin{example}[$Loop_2^n(X/\!\!/G)$]
The $n-$th power $Loop_2^n(X/\!\!/G)$ is the groupoid with objects $(\sigma, \gamma)$
where $\sigma=(\sigma_1, \sigma_2, \cdots \sigma_n)\in G^n_f$ and $\gamma: \mathbb{R}^n\longrightarrow X$ is a continuous map
such that $\gamma(s_1, \cdots, s_{i}+1, \cdots, s_n)=\gamma(s_1, \cdots, s_{i}, \cdots, s_n)\cdot \sigma_i(s_1, \cdots, s_n)$, for each $i=1, 2, \cdots n$ and $s_1, s_2 \cdots s_n\in \mathbb{R}$.
A morphism $\beta: (\sigma, \gamma)\longrightarrow (\sigma', \gamma')$ is a continuous map $\beta: \mathbb{R}^n\longrightarrow G$ satisfying $\gamma'(s)= \gamma(s)\alpha(s)$.
\label{loop2n}\end{example}

We also recall the extended loop space $Loop^{ext}_2(X/\!\!/G)$ defined in Definition 2.9 \cite{Huansurvey}.
\begin{definition}[$Loop^{ext}_2(X/\!\!/G)$]\label{loopext3space}

Let $Loop^{ext}_2(X/\!\!/G)$ denote the groupoid with the same objects
as $Loop_2(X/\!\!/G)$. A morphism $$(\sigma, \gamma)\longrightarrow
(\sigma', \gamma')$$ consists of the pair $(\alpha, t)$ with
$\alpha:\mathbb{R}\longrightarrow G$ a continuous map and
$t\in\mathbb{R}$ a rotation on $S^1$ satisfying
$\gamma'(s)=\gamma(s-t)\alpha(s-t)$. \end{definition}

We describe the $n-$th power of the functor $Loop^{ext}_2(-)$ explicitly in Example \ref{loop2next}.
\begin{example}[$Loop_2^{ext, n}(X/\!\!/G)$] The groupoid $Loop_2^{ext, n}(X/\!\!/G)$
has the same objects as $Loop_2^{n}(X/\!\!/G)$. A morphism $$(\sigma, \gamma)\longrightarrow
(\sigma', \gamma')$$ consists of the pair $(\beta, t)$ with $\beta: \mathbb{R}^n\longrightarrow G$
a continuous map and $t\in\mathbb{R}^n$ a rotation on $(S^1)^{\times n}$ satisfying
$\gamma'(s)=\gamma(s-t)\beta(s-t)$.
\label{loop2next}\end{example}

We formulate a skeleton for $Loop^{ext, n}_{2}(X/\!\!/G)$  below. The objects of $Loop^{ext, n}_{2}(X/\!\!/G)$ can be identified
with the space
$$\coprod\limits_{\sigma\in G^{n}_f}\mathcal{L}{_\sigma} X$$ where
$$\mathcal{L}{_\sigma} X=\{\gamma\in L^nX | \gamma(s+e_i)=\gamma(s)\sigma_i(s),\mbox{  for  each  }i=1, \cdots n \mbox{  and  }
s\in \mathbb{R}^n\}.$$

The groupoid $\mathcal{L}{_\sigma}
X/\!\!/ L_{\sigma}G$ is a full subgroupoid of $Loop_{2}^n(X/\!\!/G)$.

Then we study the morphisms in $Loop^{ext}_{n, 2}(X/\!\!/G)$.
The group $L_{\sigma}G\rtimes\mathbb{T}^{n}$ is isomorphic to
\begin{equation}
L_{\sigma} G\rtimes \mathbb{R}^n/\langle
(\sigma_1, -e_1), (\sigma_2, -e_2), \cdots (\sigma_n, -e_n)\rangle.\label{bengkui}
\end{equation}
$L_{\sigma}G\rtimes\mathbb{T}^{n}$ acts on
$\mathcal{L}{_\sigma} X$ by
\begin{equation}\delta \cdot(\gamma, t):= (s\mapsto
\delta(s+t)\cdot\gamma(s+t)), \mbox{  for any }(\gamma, t)\in
L_{\sigma}G\rtimes\mathbb{T}^{n},\mbox { and    }\delta \in
\mathcal{L}{_\sigma} X.\label{chongaction}\end{equation}
The action by $\sigma_i$ on $\mathcal{L}{_\sigma} X$ coincides with that
by $e_i\in\mathbb{R}^{n}$.

By similar discussion to Proposition 2.11 \cite{Huansurvey}, we have the skeleton of $Loop_2^{n}(X/\!\!/G)$ and that of $Loop_2^{ext, n}(X/\!\!/G)$
introduced below.
\begin{proposition}
(i) The groupoid
$$\coprod_{[\sigma]}\mathcal{L}{_\sigma}
X/\!\!/L_{\sigma}G$$ is a skeleton of
$Loop^{n}_2(X/\!\!/G)$, where the coproduct goes over conjugacy
classes in $\pi_0(G^{n}_f)$.

(ii) The groupoid
$$\mathcal{L}(X/\!\!/G):=\coprod_{[\sigma]}\mathcal{L}{_\sigma}
X/\!\!/L_{\sigma}G\rtimes(\mathbb{T}^{\times n})$$ is a skeleton of
$Loop^{ext, n}_2(X/\!\!/G)$, where the coproduct goes over conjugacy
classes in $\pi_0(G^{n}_f)$.   \label{loop1equivske}
\end{proposition}

Via a simple case with $n=2$, we start the discussion on the relation between $Loop^{ext, n}_1(X/\!\!/G)$ and $Loop^{ext, n}_2(X/\!\!/G)$.
\begin{example}[$Loop^{ext, 2}_1(X/\!\!/G)$]By the definition of $Loop^{ext, 2}_1(X/\!\!/G)$, an object in $Loop^{ext, n}_1(X/\!\!/G)$ consists of $Loop^{ext}_1(X/\!\!/G)-$principal bundle $P$, which is determined by a morphism $\sigma$ in $Loop^{ext}_1(X/\!\!/G)$, and a $Loop^{ext}_1(X/\!\!/G)-$map $f: P\longrightarrow (Loop^{ext}_1(X/\!\!/G))_0$, which is determined by a section $\gamma: \mathbb{R}\longrightarrow (Loop^{ext}_1(X/\!\!/G))_0$ with $\sigma=\gamma(t)^{-1} \gamma(t+1)$. Since $f$ is continuous, its image is contained in a single component of $(Loop^{ext}_1(X/\!\!/G))_0$, i.e. the subspace consisting of those elements $$\begin{CD} S^1 @<{a_L}<< P @>{a_R}>> X\end{CD}$$ with $a_L$ the principal $G-$bundle
corresponding to a fixed element $g\in G$. Thus, $\gamma$ can be viewed as a continuous map $\mathbb{R}^2\longrightarrow X$ satisfying $\gamma(t_1+1, t_2)=\gamma(t_1, t_2)\sigma(t_2)$ and $\gamma(t_1, t_2+1)=\gamma(t_1, t_2)\delta(t_1)$ for some $\delta: \mathbb{R}\longrightarrow (Loop^{ext}_1(X/\!\!/G))_1$. 

By the discusion above and Example \ref{loop2next}, we can formulate a fully faithful functor from $Loop^{ext, 2}_1(X/\!\!/G)$ to $Loop^{ext, 2}_2(X/\!\!/G)$.
\end{example}

By induction we have the conclusion below.
\begin{lemma}The groupoid $Loop^{ext, n}_1(X/\!\!/G)$ is isomorphic to a full subgroupoid of $Loop^{ext, n}_2(X/\!\!/G)$. \label{loop12}\end{lemma}

Then we consider a subgroupoid of $Loop^{ext, n}_{2}(X/\!\!/G)$, which is also a subgroupoid of $Loop^{ext, n}_1(X/\!\!/G)$.
Let $G^{tors}$ denote the set of torsion elements in $G$. Let $G^n_t$ denote the subset
$$\{\sigma=(\sigma_1, \cdots \sigma_n )\in G^{n}_f| \mbox{   Each  }\sigma_i\mbox{  is contant with image in  }G^{tors}\mbox{;   } [\sigma_i, \sigma_j]\mbox{   is  the   identity  element in  }G\}$$ of $G^n_f$.

\begin{definition} [$\Lambda^n(X/\!\!/G)$] \label{orbifoldloopspace}
Let $\Lambda^n(X/\!\!/G)$  denote the groupoid with the objects
$$\coprod\limits_{\sigma\in G^{n}_{t}}X^{\sigma},$$ and with
morphisms the space $\coprod\limits_{\sigma, \sigma'\in
G^{n}_t}\Lambda_G(\sigma,\sigma')\times X^\sigma$. Below is a little explanation.
The space $X^{\sigma}$ is the intersection $$\bigcap_{i=1}^n X^{\sigma_i}.$$ Each element $x\in X^{\sigma}$ can be viewed as a constant loop in
$\mathcal{L}{_\sigma}
X$, i.e. the image of the loop consists of the single point $x$.

Let $C_G(\sigma, \sigma')$ denote the set $$\{g\in G| \sigma_ig=g\sigma_i', \mbox{  for  each  } i=1, 2, \cdots n.\}.$$
Let $\Lambda_G(\sigma, \sigma')$ denote
the quotient of $C_G(\sigma, \sigma')\times \mathbb{R}^n$ under
the equivalence $$(x, t_1, \cdots t_i, \cdots t_n)\sim (\sigma_ix, t_1, \cdots t_i-1, \cdots t_n)=(x\sigma_i',
t_1, \cdots t_i-1, \cdots t_n)$$ for each $i$.

For $x\in X^{\sigma}$, $[a, t]\in\Lambda_G(\sigma,
\sigma')$,
\begin{equation}x\cdot ([a, t], x) := x\cdot a \in X^{\sigma'}.\label{actlt}\end{equation}

\label{olpwelldefined}\end{definition}

\begin{remark}Note that the functor $\Lambda^n(-)$ defined on global quotients is indeed the $n-$th power of the functor $\Lambda(-)$. The functor $\Lambda(-)$ is defined in Definition 2.14 \cite{Huansurvey}, from which quasi-elliptic cohomology is constructed. As shown in Definition 4.4 \cite{Huansurvey}, $\Lambda(-)$
can be defined for any orbifold groupoid. So $\lambda^n(-)$ is well-defined.
\end{remark}

%\subsection{Ghost loop}

\section{The Quasi-theory $QE_{n, G}^*(-)$}\label{definequasi}

In this section we define the quasi-theories.

Let $G$ be a compact Lie group and $n$ denote a positive integer. Let $G^{tors}_{conj}$ denote a set of representatives of
$G-$conjugacy classes in $G^{tors}$. Let $G^{n}_{z}$ denote set
$$\{\sigma=(\sigma_1, \sigma_2, \cdots \sigma_n )| \sigma_i\in G^{tors}_{conj}, [\sigma_i, \sigma_j]\mbox{   is  the   identity  element in  }G\}.$$

Let $\sigma=(\sigma_1, \sigma_2, \cdots \sigma_n)\in G^n_z$.  Define
\begin{align}C_G(\sigma)&:=C_G(\sigma, \sigma); \label{Csigmadef}\\ \Lambda_G(\sigma)&:= C_G(\sigma)\times
\mathbb{R}^n/\langle (\sigma_1, -e_1), (\sigma_2, -e_2), \cdots (\sigma_n, -e_n)\rangle=\Lambda_G(\sigma, \sigma).\label{lambdadef}\end{align}

Let $q:\mathbb{T}\longrightarrow U(1)$ denote the representation $t\mapsto e^{2\pi i t}$. Let $q_i=1\otimes\cdots\otimes q\otimes\cdots\otimes 1: \mathbb{T}^{n}\longrightarrow U(1)$ denote the tensor product with $q$ at the $i-$th position and trivial representations at other position. The representation ring $$R(\mathbb{T}^{n})\cong R(\mathbb{T})^{\otimes n}=\mathbb{Z}[q_1^{\pm}, \cdots q_n^{\pm}].$$

We have the exact sequence  \begin{equation}1\longrightarrow C_G(\sigma)\longrightarrow \Lambda_G(\sigma)\buildrel{\pi}\over\longrightarrow \mathbb{T}^{n}\longrightarrow 0 \end{equation}
where the first map is $g\mapsto [g, 0]$  and the second map is $\pi([g, t_1, \cdots t_n])=(e^{2\pi i t_1}, \cdots e^{2\pi it_n}).$
Then the map $\pi^*: R(\mathbb{T}^{n})\longrightarrow R\Lambda_G(\sigma)$ equips the representation ring $R\Lambda_G(\sigma)$ the structure as an
 $R(\mathbb{T}^{n})-$module.

We have a generalization of Lemma 3.1 \cite{Huansurvey} presenting the relation between $RC_G(\sigma)$ and $R\Lambda_G(\sigma)$.

\begin{lemma} $\pi^*: R(\mathbb{T}^{n})\longrightarrow R\Lambda_G(\sigma)$ exhibits $R\Lambda_G(\sigma)$ as a free $R(\mathbb{T}^{n})-$module.

There is an $R(\mathbb{T}^{n})-$basis of $R\Lambda_G(\sigma)$
given by irreducible representations $\{V_{\lambda}\}$, such that
restriction $V_{\lambda}\mapsto V_{\lambda}|_{C_G(\sigma)}$ to $C_G(\sigma)$
defines a bijection between $\{V_{\lambda}\}$ and the set
$\{\lambda\}$ of irreducible representations of
$C_G(\sigma)$.\label{cl}\end{lemma}

\begin{proof}
The proof of Lemma \ref{cl} is analogous to that of Lemma 3.1 \cite{Huansurvey}. I sketch it below.

Via each $C_G(\sigma)-$representation $\lambda: C_G(\sigma)\longrightarrow GL(m, \mathbb{C})$ with representation space $V$ and irreducible representations
$\eta_i: \mathbb{R}\longrightarrow GL(m, \mathbb{C})$ of $\mathbb{R}$, $i=1, 2, \cdots n$, such that $\lambda(\sigma_i)$ acts on $V$ via the scalar
multiplication by $\eta_i(1)$, we get an $m-$dimensional $\Lambda_G(\sigma)-$representation $\lambda\odot_{\mathbb{C}}(\eta_1\otimes\cdots\otimes \eta_n)$ with representation space $V$. Let $l_i$ denote the order of the element $\sigma_i$ and $\eta_i(1) =e^{\frac{2\pi i k_i}{l_i}}$ for some $k_i\in\mathbb{Z}$. Since each $\eta_i$ is $1-$dimensional, for any $t\in \mathbb{R}$, $$\eta_i(t)= e^{2\pi i(\frac{k_i}{l_i}+m_i)t}$$ for some integer $m_i$. We have
\begin{equation}\lambda\odot_{\mathbb{C}}(\eta_1\otimes\cdots\otimes \eta_n)([g, t_1,\cdots t_n])=\lambda(g)\eta_1(t_1)\cdots\eta_n(t_n)\end{equation}

In the other direction, given an irreducible $\Lambda_G(\sigma)-$representation $\rho$ with representation space $W$, its restriction to each factor gives an irreducible $C_G(\sigma)-$representation $\lambda$ with underlying space $V$ and an irreducible representation $\eta$ of $\mathbb{R}^n$, thus irreducible representations $\{\eta_i\}_i$ of $\mathbb{R}$. They satisfy $\lambda(\sigma_i)=\eta_i(1)I$.

It is straightforward to check the conclusion is true.
\end{proof}

\begin{definition}For equivariant cohomology theories $\{E_{H}^*\}_H$ and any $G-$space $X$, the corresponding quasi-theory $QE_{n, G}^*(X)$ is defined to be
$$\prod_{\sigma\in
G^{n}_{z}}E^*_{\Lambda_G(\sigma)}(X^{\sigma}).$$\label{qedef}\end{definition}

\begin{remark} If there is an orbifold theory $E_{orb}^*$ satisfying $E_{orb}^*(X/\!\!/G)=E_{G}^*(X)$ for each $G-$space $X$, then Definition \ref{qedef}
can be expressed as  $$QE_{n, G}^*(X)\cong E^*_{orb}(\Lambda^n(X/\!\!/G)),$$
where the groupoid $\Lambda(X/\!\!/G)$ is defined in Definition \ref{olpwelldefined}.\end{remark}

\begin{example}[Motivating example: Tate K-theory  and quasi-elliptic cohomology]
Tate $K-$theory is the generalized elliptic cohomology associated to the Tate curve. Its divisible group is $\mathbb{G}_m\oplus \mathbb{Q}/\mathbb{Z}$.
In \cite{Huansurvey} we introduce quasi-elliptic cohomology $QEll^*_G(-)$. It is exactly the theory $QK_{1, G}^*(-)$ in Definition \ref{qedef}.
We have the relation
\begin{equation}QEll^*_G(X)\otimes_{\mathbb{Z}[q^{\pm}]}\mathbb{Z}((q))=(K^*_{Tate})_G(X)
\label{tateqellequiv}\end{equation}
\end{example}

\begin{example}[Generalized Tate K-theory and generalized quasi-elliptic cohomology]

In Section 2 \cite{Gan07} Ganter  gave an interpretation of $G-$equivariant Tate K-theory for finite groups $G$
by the loop space of a global quotient orbifold. Apply the loop construction $n$ times, we can get the $n-$th generalized Tate K-theory. The divisible group associated to it is
$\mathbb{G}_m\oplus (\mathbb{Q}/\mathbb{Z})^n$.

With quasi-theories, we can get a neat expression of it. Consider the quasi-theory
$$QK_{n, G}^*(X)=\prod_{\sigma\in
G^{n}_{z}}K^*_{\Lambda_G(\sigma)}(X^{\sigma}).$$
$QK_{n, G}^*(X)\otimes_{\mathbb{Z}[q^{\pm}]^{\otimes n}}\mathbb{Z}((q))^{\otimes n}$ is isomorphic to the $n-$th generalized Tate K-theory.

The theories $QK_{n, G}^*(-)$ has all the properties and features that we cited in Section \ref{properties}.
\label{generalizedquasi}
\end{example}

\section{Properties}\label{properties}

In this section we present some properties of the quasi-theories. 

Since each homomorphism $\phi: G\longrightarrow H$ induces a
well-defined homomorphism $\phi_{\Lambda}:
\Lambda_G(\tau)\longrightarrow\Lambda_H(\phi(\tau))$ for each
$\tau$ in $G^{n}_{z}$, we can get the proposition below directly.
\begin{proposition}If $\{E^*_G(-)\}_G$  have restriction maps, then each homomorphism $\phi: G\longrightarrow H$ induces a ring map
$$\phi^*: QE^*_{n, H}(X)\longrightarrow QE^*_{n, G}(\phi^*X)$$ characterized by the commutative diagrams

\begin{equation}\begin{CD}QE^*_{n, H}(X) @>{\phi^*}>> QE^*_{n, G}(\phi^*X) \\ @V{\pi_{\phi(\tau)}}VV  @V{\pi_{\tau}}VV  \\
E^*_{\Lambda_H(\phi(\tau))}(X^{\phi(\tau)}) @>{\phi^*_{\Lambda}}>>
E^*_{\Lambda_G(\tau)}(X^{\phi(\tau)})\end{CD}\end{equation} for any $\tau \in G^{n}_{z}$. So $QE^*_G$ is functorial in $G$.
\label{restrictionq}\end{proposition}

Next we show if $\{E_{G}^*\}_G$ has the change-of-group isomorphism, then so does $\{QE^*_{n, G}(-)\}_G$.

Let $H$ be a closed subgroup of $G$ and $X$ a $H$-space. Let
$\phi: H\longrightarrow G$ denote the inclusion homomorphism. The
change-of-group map $\rho^G_H: QE^*_{n, G}(G\times_HX)\longrightarrow
QE^*_{n, H}(X)$ is defined as the composite
\begin{equation}QE^*_{n, G}(G\times_HX)\buildrel{\phi^*}\over\longrightarrow
QE^*_{n, H}(G\times_H X)\buildrel{i^*}\over\longrightarrow
QE_{n, H}^*(X)\label{changeofgroup}
\end{equation}
where $\phi^*$ is the restriction map and $i: X\longrightarrow
G\times_HX$ is the $H-$equivariant map defined by $i(x)=[e, x].$

\begin{proposition}If $\{E_{G}^*\}_G$ has the change-of-group isomorphism, the change-of-group map
$$\rho^G_H: QE^*_{n, G}(G\times_H X)\longrightarrow
QE^*_{n, H}(X)$$ defined in (\ref{changeofgroup}) is an
isomorphism.\end{proposition}
\begin{proof}

The proof is analogous to that of Proposition 3.19 \cite{Huansurvey}. I sketch
it below.

For any $\tau\in H^{n}_{z}$, there exists a unique
$\sigma_{\tau}\in G_{z}^{n}$ such that
$\tau=g_{\tau}\sigma_{\tau}g_{\tau}^{-1}$ for some $g_{\tau}\in
G$.  Consider the maps \begin{equation}\begin{CD}
\Lambda_G(\tau)\times_{\Lambda_H(\tau)}X^{\tau}@>{[[a, t], x
]\mapsto [a, x]}>> (G\times_H X)^{\tau}@>{[u, x]\mapsto
[g_{\tau}^{-1}u, x]}>>
(G\times_HX)^{\sigma}.\end{CD}\end{equation} The first map is
$\Lambda_G(\tau)-$equivariant and the second is equivariant with
respect to the homomorphism $c_{g_{\tau}}:
\Lambda_{G}(\sigma)\longrightarrow \Lambda_G(\tau)$ sending $[u,
t]\mapsto [g_{\tau} u g_{\tau}^{-1}, t]$. Taking a coproduct over
all the elements $\tau\in H^{n}_{z}$ that are conjugate to
$\sigma\in G^{n}_{z}$ in $G$, we get an isomorphism
$$\gamma_{\sigma}: \coprod_{\tau}\Lambda_G(\tau)\times_{\Lambda_H(\tau)} X^{\tau}\longrightarrow
(G\times_HX)^{\sigma}$$ which is $\Lambda_G(\sigma)-$equivariant
with respect to $c_{g_{\tau}}$. Then we have the map
\begin{equation}\gamma:=\prod_{\sigma\in G^{n}_{z}}\gamma_{\sigma}:
\prod_{\sigma\in G^{n}_{z}}
E^*_{\Lambda_G(\sigma)}(G\times_HX)^{\sigma}\longrightarrow
\prod_{\sigma\in
G^{n}_{z}}E^*_{\Lambda_G(\sigma)}(\coprod_{\tau}\Lambda_G(\tau)\times_{\Lambda_H(\tau)}
X^{\tau})
\end{equation}

It's straightforward to check the change-of-group map coincide
with the composite \begin{align*} QE^*_{n, G}(G\times_H
X)\buildrel{\gamma}\over\longrightarrow \prod_{\sigma\in
G^{n}_{z}}E^*_{\Lambda_G(\sigma)}(\coprod_{\tau}\Lambda_G(\tau)\times_{\Lambda_H(\tau)}
X^{\tau})\longrightarrow &\prod_{\tau\in
H^{n}_{z}}E^*_{\Lambda_H(\tau)}(X^{\tau})\\&=QE^*_{n, H}(X)\end{align*}
with  the second map  the change-of-group isomorphism in
 $\{E^*_{n, G}(-)\}$.
\end{proof}

Next we show that %if $E$ has K\"{u}nneth map, then so do
the theories $QK^*_{n, G}(-)$ have K\"{u}nneth map.

Let $G$ and $H$ be two compact Lie groups. $X$ is a $G$-space and
$Y$ is a $H$-space. Let $\sigma\in G^{n}_z$ and $\tau\in
H^{n}_z$. Let
$\Lambda_G(\sigma)\times_{\mathbb{T}^{n}}\Lambda_H(\tau)$ denote the
fibered product of the morphisms
$$\Lambda_G(\sigma)\buildrel{\pi}\over\longrightarrow
\mathbb{T}^n\buildrel{\pi}\over\longleftarrow\Lambda_H(\tau).$$ It is
isomorphic to $\Lambda_{G\times H}(\sigma, \tau)$ under the
correspondence
$$([\alpha, t], [\beta, t])\mapsto [\alpha, \beta, t].$$

Consider the map below
\begin{align*}T: K_{\Lambda_G(\sigma)}(X^{\sigma})\otimes
K_{\Lambda_H(\tau)}(Y^{\tau})&\longrightarrow
K_{\Lambda_{G}(\sigma)\times\Lambda_{H}(\tau)}(X^{\sigma}\times
Y^{\tau})\buildrel{res}\over\longrightarrow
K_{\Lambda_{G}(\sigma)\times_{\mathbb{T}^n}\Lambda_{H}(\tau)}(X^{\sigma}\times
Y^{\tau})\\ &\buildrel{\cong}\over\longrightarrow
K_{\Lambda_{G\times H}(\sigma, \tau)}((X\times Y)^{(\sigma,
\tau)}).\end{align*} where the first map is the K\"{u}nneth map of
equivariant K-theory, the second is the restriction map  and the
third is the isomorphism induced by the group isomorphism
$\Lambda_{G\times H}(\sigma,
\tau)\cong\Lambda_G(\sigma)\times_{\mathbb{T}^n}\Lambda_H(\tau)$.

For $\sigma\in G^{n}_z$, let $1$ denote
the trivial line bundle over $X^{\sigma}$ and let $q_i$ denote the line
bundle $1\odot q_i$ over $X^{\sigma}$. The map $T$ above sends both
$1\otimes q_i$ and $q_i\otimes 1$  to $q_i$. So we get the well-defined
map
\begin{equation}K^*_{\Lambda_G(\sigma)}(X^{\sigma})\otimes_{\mathbb{Z}[q^{\pm}]^{\otimes n}}K^*_{\Lambda_H(\tau)}(Y^{\tau})\longrightarrow
K_{\Lambda_{G\times H}(\sigma, \tau)}((X\times
Y)^{(\sigma, \tau)}).\label{ku}\end{equation}

\begin{definition}The tensor product of the quasi-theory $\{QK^*_{n, G}(-)\}_G$ is defined by
\begin{equation}QK^*_{n, G}(X)\widehat{\otimes}_{\mathbb{Z}[q^{\pm}]^{\otimes n}}QK^*_{n, H}(Y)
\cong\prod_{\sigma\in G^{n}_{z}\mbox{,   } \tau\in
H^{n}_{z}}K^*_{\Lambda_G(\sigma)}(X^{\sigma})\otimes_{\mathbb{Z}[q^{\pm}]^{\otimes n}}K^*_{\Lambda_H(\tau)}(Y^{\tau}).\label{qectensor}\end{equation}
The direct product of the maps defined in (\ref{ku}) gives a ring
homomorphism
$$QK^*_{n, G}(X)\widehat{\otimes}_{\mathbb{Z}[q^{\pm}]^{\otimes n}}QK^*_{n, H}(Y)\longrightarrow
QK^*_{n, G\times H}(X\times Y),$$ which is the K\"{u}nneth map of the quasi-theory $\{QK^*_{n, G}(-)\}_G$.
\end{definition}

In addition, as quasi-elliptic cohomology, the quasi-theories $\{QK^*_{n, G}(-)\}_G$ inherits the properties from equivariant K-theories.
The proof of them are similar to that of Proposition 3.17, Proposition 3.18 in \cite{Huansurvey}.

By Lemma \ref{cl} we have
$$QK^*_{n, G}(\mbox{pt})\widehat{\otimes}_{\mathbb{Z}[q^{\pm}]^{\otimes n}}QK^*_{n, H}(\mbox{pt})=QK^*_{n, G\times H}(\mbox{pt}).$$

\begin{proposition}
Let $X$ be a $G\times
H-$space with trivial $H-$action and let $\mbox{pt}$ be the single
point space with trivial $H-$action.
Then we have
$$QK^*_{n, G\times H}(X)\cong QK^*_{n, G}(X)\widehat{\otimes}_{\mathbb{Z}[q^{\pm}]^{\otimes n}} QK^*_{n, H}(pt).$$

Especially, if $G$ acts trivially on $X$, we have
$$QK^*_G(X)\cong   QK^*_{n, \{e\}}(X)\widehat{\otimes}_{\mathbb{Z}[q^{\pm}]^{\otimes n}}
QK^*_{n, G}(\mbox{pt}).$$ Here $\{e\}$ is
the trivial group.
\end{proposition}

\begin{proposition}
If $G$ acts freely on $X$,
$$QEll^*_G(X)\cong QEll^*_e(X/G).$$
\end{proposition}

\begin{remark}
If the equivariant cohomology theories $\{E^*_{G}(-)\}_G$ have those good properties as equivariant K-theories, then
they share corresponding properties as quasi-elliptic cohomology theory, all the properties in Section 3.3 \cite{Huansurvey}.
\end{remark}

%\section{Quasi-theories}In 2005 Lurie
%announced a result that gave sufficient conditions to functorially realize a family of 1-dimensional formal group laws by spectra given certain properties and certain
%extra data, i.e. that from a $p-$divisible group (or Barsotti Tate group). In Theorem 9 \cite{Lawsonvar} Lawson states a version of this Lurie's result.

%\begin{proposition}Let $E$ denote a cohomology theory with formal group $\mathbb{G}_E$. If it satisfies \begin{equation}\mbox{Spec}(E^0(BZ_p))=\mathbb{G}_E(Z_p)\end{equation} and %the $p-$divisible group associated to $QE_{1, G}^*(-)\otimes_{\mathbb{Z}[q^{\pm}]}\mathbb{Z}((q))$
%is $\mathbb{G}_E+Q_p/Z_p$, then for any positive integer $n$, $QE_{n, G}^*(-)\otimes_{\mathbb{Z}[q^{\pm}]^{\otimes n}}\mathbb{Z}((q))^{\otimes n}$ has $p-$divisible group %$\mathbb{G}_E+(Q_p/Z_p)^n$.\label{quasimean}\end{proposition}

%\begin{remark}Proposition \ref{quasimean} explains why we call the theories $QE_{n, G}^*(-)$ "quasi-theories".\end{remark}

\end{document}